\documentclass{mfat}
\pagespan{62}{73}

\DeclareMathOperator*{\TIMES}{\LARGE\pmb\times}
\DeclareMathOperator*{\Times}{\Large\pmb\times}

\newtheorem{theorem}{Theorem}[section]
\newtheorem{proposition}[theorem]{Proposition}
\theoremstyle{remark}
\newtheorem{remark}[theorem]{Remark}
\newtheorem*{example}{Example}

\begin{document}

\title[Joint functional calculus in algebra of polynomial distributions]
    {Joint functional calculus in algebra of polynomial tempered distributions}

\author{S. V. Sharyn}
\address{Department of Mathematics and Computer Sciences, Precarpathian National University,
           57 Shevchenka str., Ivano-Frankivsk, 76018, Ukraine}
\email{sharyn.sergii@gmail.com}

\date{29/12/2014;\ \  Revised 21/04/2015}

\begin{abstract}
  In this paper we develop a functional calculus for a countable system
  of generators of contraction strongly continuous semigroups. As a
  symbol class of such calculus we use the algebra of polynomial
  tempered distributions. We prove a differen\-tial property of
  constructed calculus and describe its image with the help of the
  commutant of polynomial shift semigroup. As an application, we
  consider a function of countable set of second derivative operators.
\end{abstract}

\maketitle

\section*{Introduction}

A functional calculus is a theory that studies how to construct
functions depending on operators (roughly speaking, how to
``substitute'' an operator instead of the variable in a function).
Also it is said that the functional calculus for some (not necessary
bounded) operator $A$ on a Banach space is a method of associating an
operator $f(A)$ to a function $f$ belonging to a topological algebra
$\mathcal A$ of functions. If we have such a method then, actually, we
have a continuous homomorphism from the algebra $\mathcal A$ to a
topological algebra of operators.  So, in this terminology the
functional calculus can be identified with the above-mentioned
homomorphism (but as a theory the functional calculus studies such
homomorphisms).

There are many different approaches to construct a functional
calculus for one ope\-rator acting on a Banach space. For
Riesz-Dunford functional calculus, based on the Cauchy formula, we
refer the reader to the book \cite{Haase}. Such a functional
calculus has applications, in particular, in the spectral theory
of elliptic differential equations and maximal regularity of
parabolic evolution equations (see e.g. \cite{Kalton, Kriegler}).
Another method, based on the Laplace transformation, was developed
in \cite{HillePhillips57}. This method is known as the
Hille-Phillips functional calculus. It has many helpful
applications, in particular, in hydrology (see \cite{HaaseAppl}
and the references given there). Such type of calculus is the main
object of investigation in this article.

The Hille-Phillips functional calculus for functions of several
variables is well developed (see e.g. \cite{LopSharSMJ,Mirotin}).
The case of functions of infinitely many variables is less
studied. We mention the book \cite{Samoilenko} that is devoted to
spectral questions (among them there is a functional calculus) of
countable families of self-adjoint operators on a Hilbert space.
The main goal of this article is the construction of
Hille-Phillips type functional calculus for countable set of
generators of contraction strongly continuous semigroups, acting
on a Banach space.

The Borchers-Uhlmann algebra, i.e. the tensor algebra over the
space of rapidly decreasing functions with tensor product as a
multiplication was effectively used in quantum field theory (see
e.g. \cite{Borchers88,Borchers62,Uhlmann}). Such algebras have an
equivalent structure of polynomials with pointwise multiplication
\cite{Din99}. It was an incitement to research the problems
connected with the polynomially extended cross-correlation of
(ultra)distributions, differentiations and the corresponding
functional calculus \cite{LopSharTopology}. Elements of the
Borchers-Uhlmann algebra can be treated as functionals on spaces
of smooth functions of infinite many variables. So, we can
understand this algebra as space of polynomial distributions with
tensor structure. In this paper we would like to consider this
structure for a special case.

A Fr\'{e}chet-Schwartz space (briefly, $(FS)$ space) is one that
is Fr\'{e}chet and Schwartz simultaneously (see \cite{Zharinov}).
Let $\mathcal S_+$ be the space of rapidly decreasing functions on
$[0,+\infty)$ and $\mathcal S'_+$ be its dual. It is known (see
e.g. \cite{Komatsu,Smirnov}) that these spaces are nuclear
Fr\'{e}chet-Schwartz and dual Fr\'{e}chet-Schwartz spaces ($(DFS)$
for short), respectively. These facts are crucial for our
investigation. The main objects of investigation are the algebras
${\mathcal P}(\mathcal S'_+)$ and ${\mathcal P}'(\mathcal S'_+)$
of polynomial test and generalized functions, which have the
tensor structures of the forms
$\bigoplus\nolimits_{\operatorname{fin}}\mathcal S_+^{\hat\otimes
n}$ and $\Times\mathcal S_+^{\prime\hat\otimes n}$, respectively.

Using the Grothendieck technique \cite{Grot55}, we introduce the
polynomial extension of cross-correlation and prove the Theorem
\ref{theorCrossCorel} about isomorphic representation of the
algebra of polynomial distributions onto the commutant of
polynomial shift semigroup (see \eqref{polynomialshift}) in the
space of linear continuous operators on
$\bigoplus_{\operatorname{fin}}\mathcal S_+^{\hat\otimes n}$. In
Proposition \ref{prop33} we prove the differential property of
polynomial cross-correlation, which is essentially used in main
Theorem \ref{main}.

In the section \ref{sec.alebra} we extend the generalized Fourier
transformation onto the spaces of polynomial test and generalized
functions. Images of this map we understand as functions and
functionals of infinite many variables (see Remarks
\ref{infinitemanyvar} and \ref{rem2}), respectively.

The constructed polynomial test and generalized functions we apply to
an operator semigroup with infinitely many parameters. Namely, we
construct the functional calculus for countable system of generators
of contraction $C_0$-semigroups and prove its properties (see Remark
\ref{rem41} and Theorem \ref{main}).  This calculus is an
infinite-dimensional analogue of the one constructed in
\cite{LopSharSMJ}.  As an example we consider the infinite-dimensional
Gaussian semigroup, which is generated by a countable set of second
derivative operators.

Finally we note that there are other known and widely used
infinite-dimensional ge\-ne\-ralizations of classical spaces of
distributions \cite{BerKond,BerSam}. For example, white noise
analysis is based on an infinite dimensional analogue of the
Schwartz distribution theory (see e.g.
\cite{Hida_Kuo_Potthoff_Streit93,Kachanovsky,KONDRATIEV_STREIT_WESTERKAMP_Yan95,QuantumField}).

\section{Background on polynomial tempered distributions}

In what follows ${L}(X,Y)$ denotes the space of all continuous linear
operators from a locally convex space $X$ into another such space $Y$,
endowed with the topology of uniform convergence on bounded subsets of
$X$. Let ${L}(X):={L}(X,X)$ and $I_X$ denotes the identity operator in
${L}(X)$. The dual space $X':={L}(X,\mathbb C)$ is endowed with strong
topology. The pairing between elements of $X'$ and $X$ we denote by
$\langle \,\cdot\,, \,\cdot\,\rangle$.

Let $X^{\hat\otimes n}$, $n\in\mathbb N$, be the symmetric $n$th
tensor degree of $X$, completed in the projective tensor topology.
For any $x\in X$ we denote $x^{\otimes n}:=
\underbrace{x\otimes\dots\otimes x}_n\in X^{\hat\otimes n}$,
$n\in\mathbb N$. Set $X^{\hat\otimes 0}:=\mathbb C$, $x^{\otimes
0}:=1\in\mathbb C$.

For any $A\in L(X)$ its tensor power $A^{\otimes n}\in
L(X^{\hat\otimes n})$, $n\in\mathbb N$, is defined as a linear
continuous extension of the map $x^{\otimes n} \longmapsto
(Ax)^{\otimes n}$, where $x\in X$.  It follows from results of
\cite{BSU96} that such an extension exists if $X$ is a projective or
inductive limit of separable Hilbert spaces.  In this article we
consider only such spaces.

Let $\mathcal S_+$ be the Schwartz space of rapidly decreasing functions
on $\mathbb R_+:=[0,+\infty)$ and  $\mathcal S'_+$ be
its dual space of tempered distributions supported by $\mathbb R_+$.
Note that strong topology on $\mathcal S'_+$ coincides
with the Mackey topology and topology of inductive limit (see \cite[IV.4, IV.5]{Schaefer}).
Let $\delta_t$ be the Dirac delta functional concentrated
at a point $t\in\mathbb R_+$.
It is known \cite{Vladimirov} that $\mathcal S'_+$
is a topological algebra with the unit $\delta:=\delta_0$
under the convolution, defined as
\[
\langle f*g, \varphi\rangle = \langle f(s), \langle g(t),
\varphi(s+t) \rangle \rangle,\quad f, g\in\mathcal S'_+,\quad
\varphi\in\mathcal{S}_+.
\]
Note that here and everywhere the notation $f(t)$ shows that
a functional $f$ acts on a test function in the variable $t$.

From the duality theory as well as from the theory of nuclear spaces
it follows that $\mathcal S_+$ is a nuclear $(FS)$ space, and
$\mathcal S'_+$ is a nuclear $(DFS)$ space.

To define the locally convex space ${\mathcal P}({}^n\!\mathcal S'_+)$ of $n$-homogeneous
polynomials on $\mathcal S'_+$ we use the canonical topological linear isomorphism
\[
{\mathcal P}({}^n\!\mathcal S'_+)\simeq (\mathcal S_+^{\prime\hat\otimes n})'
\]
described in \cite{Din99}. Namely, given a functional
$p_n\in(\mathcal S_+^{\prime\hat\otimes n})'$, we define an
$n$-homogeneous polynomial $P_n\in{\mathcal P}({}^n\!\mathcal S'_+)$ by
$P_n(f):=p_n(f^{\otimes n})$,
$f\in \mathcal S'_+$.
We equip ${\mathcal P}({}^n\!\mathcal S'_+)$ with the locally convex topology $\mathfrak{b}$ of
uniform convergence on bounded sets in $\mathcal S'_+$. Set
${\mathcal P}({}^0\!\mathcal S'_+):=\mathbb{C}$.
The space $\mathcal{P}(\mathcal S'_+)$ of all continuous polynomials on $\mathcal S'_+$ is
defined to be the complex linear span of all ${\mathcal P}({}^n\!\mathcal S'_+)$, $n\in\mathbb Z_+$,
endowed with  the topology $\mathfrak{b}$. Let ${\mathcal P}'(\mathcal S'_+)$ mean the strong dual of
$\mathcal{P}(\mathcal S'_+)$ and
\[
\Gamma(\mathcal S_+):=
\mathop{\bigoplus\nolimits_{\operatorname{fin}}}\limits_{\hskip-10pt
n\in\mathbb Z_+}\mathcal S_+^{\hat\otimes n} \subset
\bigoplus_{n\in\mathbb Z_+}\mathcal S_+^{\hat\otimes n}
\quad\text{and}\quad \Gamma(\mathcal S'_+):=\TIMES_{n\in\mathbb
Z_+}\mathcal S_+^{\prime\hat\otimes n}.
\]
Note that we consider only the case when elements of the direct sum
consist of a finite but not fixed number of terms. For simplicity of
notation we write $\Gamma(\mathcal S_+)$ instead of commonly used
$\Gamma_{\operatorname{fin}}(\mathcal S_+)$.

We have the following assertion (see also \cite[Proposition
2.1]{LopSharTopology}).

\vspace*{-2pt}
\begin{proposition}\label{propTensStruct}
There exist linear topological isomorphisms
\[
\Upsilon : {\mathcal P}(\mathcal S'_+) \longrightarrow
\Gamma(\mathcal S_+), \quad \Psi : {\mathcal P}'(\mathcal S'_+)
\longrightarrow \Gamma(\mathcal S'_+).
\]
\end{proposition}

Elements of the spaces
${\mathcal P}(\mathcal S'_+)$ and ${\mathcal P}'(\mathcal S'_+)$
we call the polynomial test functions and polynomial distributions, respectively.
In what follows elements of the spaces $\Gamma(\mathcal S_+)$ and $\Gamma(\mathcal S'_+)$
will be written as
\[
\bigoplus_{n=0}^m
p_n=(p_0,p_1,\dots,p_m,0,\dots)\in\Gamma(\mathcal S_+)
\quad\text{and}\quad \TIMES_{n\in\mathbb Z_+}
f_n=(f_0,f_1,\dots,f_n,\dots)\in\Gamma(\mathcal S'_+)
\]
for some $m\in\mathbb N$,
where $p_n\in\mathcal S_+^{\hat\otimes n}$ and $f_n\in \mathcal S_+^{\prime\hat\otimes n}$
for all $n\in\mathbb Z_+$. To simplify, we write $\big(p_n\big)$ and $\big(f_n\big)$ instead of
$\bigoplus_{n=0}^m p_n$ and $\Times_{n\in\mathbb Z_+} f_n$, respectively.

Note that the following systems of elements
\begin{equation}\label{totalsystem}
\big\{\big(\varphi^{\otimes n}\big): \varphi\in\mathcal S_+\big\},
\quad \big\{\big(f^{\otimes n}\big):f\in\mathcal S'_+\big\}
\end{equation}
are total in $\Gamma(\mathcal S_+)$ and $\Gamma(\mathcal S'_+)$, respectively.

Let us define the operation
\[
\big(f^{\otimes n}\big)\circledast \big(g^{\otimes n}\big):= \big((f*g)^{\otimes n}\big)
\]
for elements from the total subset \eqref{totalsystem} of the space
$\Gamma(\mathcal S'_+)$ and extend it to the whole space by linearity
and continuity.  It is obvious that $\Gamma(\mathcal S'_+)$ is an
algebra relative to the operation $\circledast$ with the unit element
$\big(\delta^{\otimes n}\big)$. Since $\Gamma(\mathcal S_+)$ is
continuously and densely embedded into $\Gamma(\mathcal S'_+)$ (see
\cite{LopSharTopology}) and the space $\mathcal S_+$ is a convolution
algebra (see \cite{Vladimirov}), the space $\Gamma(\mathcal S_+)$
becomes an algebra with respect to $\circledast$.

For any $K\in {L}(\mathcal{S}_+)$ let us define an operator
$K^{\otimes}\in{L}\big(\Gamma(\mathcal S_+)\big)$ as follows:
\begin{equation}\label{Kotimes}
K^{\otimes}:=\big(K^{\otimes n}\big):
\boldsymbol p=(p_n)
\quad\longmapsto\quad
K^{\otimes}\boldsymbol p:=\big(K^{\otimes n}p_n\big),
\end{equation}
where $K^{\otimes 0}:=I_{\mathbb C}$ and each operator $K^{\otimes
n}\in{L}(\mathcal S_+^{\hat\otimes n})$ is defined as a linear
continuous extension of the map $\varphi^{\otimes n} \longmapsto
(K\varphi)^{\otimes n}$, with $\varphi\in \mathcal S_+$,
$n\in\mathbb N$.

Consider the one-parameter $C_0$-semigroup of shifts,
\begin{equation*}
T\colon\mathbb{R}_+\ni s\longmapsto T_s\in{L}(\mathcal{S}_+),
\quad
T_s\varphi(t):=\varphi(t+s),\quad t\in\mathbb R_+,
\quad \varphi\in\mathcal{S}_+.
\end{equation*}
Hence, the map $T^{\otimes n}: \mathbb{R}_+\ni s\longmapsto
T_s^{\otimes n}\in {L}(\mathcal S_+^{\hat\otimes n})$ is well defined.
It easy to check that $T^{\otimes n}$ is a one-parameter semigroup of
operators.  Denote $T^{\otimes}_s:= \big(T^{\otimes n}_s\big)$,
$s\in\mathbb R_+$. The mapping
\begin{equation}\label{polynomialshift}
T^{\otimes}:\mathbb R_+\ni s\longmapsto T^{\otimes}_s\in{L}\big(\Gamma(\mathcal S_+)\big)
\end{equation}
is called the polynomial shift semigroup.

The cross-correlation of a distribution $f\in\mathcal S'_+$ and
a function $\varphi\in\mathcal S_+$ is defined to be
the function
\[(f\star\varphi)(s):=\langle f, T_s\varphi\rangle=
\langle f(t), \varphi(t+s)\rangle.\]
Similarly to \cite{SharCrosCor} it is easy to prove that
\begin{equation}\label{fgstarphi}
f\star\varphi\in\mathcal S_+,\quad f\star
T_s\varphi=T_s(f\star\varphi)\quad\text{and}\quad
(f*g)\star\varphi=f\star(g\star\varphi)
\end{equation}
for any $s\in\mathbb R_+$, $f,g\in\mathcal S'_+$ and $\varphi\in\mathcal S_+$.
It follows that the cross-correlation operator
defined by \[K_f\colon\varphi\longmapsto f\star\varphi\]
belongs to $ L(\mathcal S_+)$ for any $f\in\mathcal S'_+$.
From \eqref{Kotimes} it follows that
\begin{equation}\label{Kfotimesn}
K_{\boldsymbol f}^\otimes:=\big(K_{f_n}^{\otimes n}\big)\in
{L}\big(\Gamma(\mathcal S_+)\big)\quad
\text{and}\quad K_{f_n}^{\otimes n}\in{L}(\mathcal S_+^{\hat\otimes n}),
\end{equation}
where
$\boldsymbol f:=(f_n)\in\Gamma(\mathcal S'_+)$ with
$f_n\in \mathcal S_+^{\prime\hat\otimes n}$,
$n\in\mathbb Z_+$.

The cross-correlation of a polynomial distribution
$\boldsymbol f=(f_n)\in\Gamma(\mathcal S'_+)$ and a polynomial test function
$\boldsymbol p=(p_n)\in\Gamma(\mathcal S_+)$ is given by
\[
\boldsymbol f\star \boldsymbol p:=
K_{\boldsymbol f}^{\otimes}\boldsymbol p=\big(K_{f_n}^{\otimes n}p_n\big).
\]

\begin{proposition}
For any $\boldsymbol f\in\Gamma(\mathcal S'_+)$ and
$\boldsymbol p\in\Gamma(\mathcal S_+)$ the cross-correlation
$\boldsymbol f\star \boldsymbol p$ is a polynomial test function belonging to
$\Gamma(\mathcal S_+)$.
\end{proposition}
\begin{proof}
  Let $\boldsymbol f=(f_n)\in\Gamma(\mathcal S'_+)$ and $\boldsymbol
  p=(p_n)\in\Gamma(\mathcal S_+)$.  Since $\boldsymbol f \star
  \boldsymbol p= \big(K_{f_n}^{\otimes n}p_n\big)$ by definition, we
  only need to check that $K_{f_n}^{\otimes n}p_n\in\mathcal
  S_+^{\hat\otimes n}$ for all $n\in\mathbb Z_+$.  In the case $n=0$
  this is obvious. If $n=1$ we obtain that $\langle f_1, T_s
  p_1\rangle=(f_1\star p_1)(s)$ belongs to $\mathcal S_+$ (see
  \eqref{fgstarphi}). Consider the case $n>1$.  Since the operators
  $K_{f_n}^{\otimes n}$ are linear and continuous, it is sufficient to
  prove the statement for $f_n=f^{\otimes n}$ and
  $p_n=\varphi^{\otimes n}$ with $f\in\mathcal S'_+$ and
  $\varphi\in\mathcal S_+$. Then the function
\[
K_{f_n}^{\otimes n}p_n=
\big\langle f^{\otimes n}, T_s^{\otimes n}{\varphi}^{\otimes n}\big\rangle=
\big\langle f^{\otimes n}, (T_s\varphi)^{\otimes n}\big\rangle=
\big\langle f, T_s\varphi\big\rangle^{\otimes n}=
(f\star\varphi)^{\otimes n}
\]
belongs to $\mathcal S_+^{\hat\otimes n}$ as the $n$-th tensor power of
a function from $\mathcal S_+$.
\end{proof}

The commutant $\big[T^{\otimes}\big]^c\subset L\big(\Gamma(\mathcal
S_+)\big)$ of the polynomial shift semigroup $T^{\otimes}$ is defined
to be the set
\[
\big[T^{\otimes}\big]^c:=\big\{K^{\otimes}\in{L}\big(\Gamma(\mathcal S_+)\big):
K^{\otimes}\circ T^{\otimes}_s=T^{\otimes}_s\circ K^{\otimes},
\forall\, s\in\mathbb R_+\big\},
\]
where $K^{\otimes}$ is defined by \eqref{Kotimes}.
\begin{theorem}\label{theorCrossCorel}
The mapping
\[
\Gamma(\mathcal S'_+) \ni \boldsymbol f
\longmapsto
K_{\boldsymbol f}^{\otimes}
\in  L\big(\Gamma(\mathcal S_+)\big)
\]
is an algebraic isomorphism from the algebra
$\big\{\Gamma(\mathcal S'_+), \circledast\big\}$ onto the
commutant $\big[T^{\otimes}\big]^c$ of the semigroup $T^{\otimes}$
in the algebra $\big\{ L\big(\Gamma(\mathcal
S_+)\big),\circ\big\}$. In particular, the following relation
holds:
\begin{equation*}
K_{\boldsymbol f\circledast \boldsymbol g}^{\otimes}=
K_{\boldsymbol f}^{\otimes}\circ K_{\boldsymbol g}^{\otimes},
\quad \boldsymbol f,\boldsymbol g\in\Gamma(\mathcal S'_+).
\end{equation*}
\end{theorem}
\begin{proof}
Since the operator $K_{\boldsymbol f}^{\otimes}$ is linear and continuous,
it is sufficient to consider only elements
from the total subsets \eqref{totalsystem}.
Let $\boldsymbol p=(\varphi^{\otimes n})\in\Gamma(\mathcal S_+)$ with
$\varphi\in\mathcal S_+$ and $\boldsymbol f=(f^{\otimes n}),
\boldsymbol g=(g^{\otimes n})\in\Gamma(\mathcal S'_+)$ with $f,g\in\mathcal S'_+$ be given.
From definitions of operations $\circledast$ and $\star$, as well as from
\eqref{fgstarphi}, we obtain
\[
\begin{split}
K_{\boldsymbol f\circledast \boldsymbol g}^{\otimes}\boldsymbol p=&
\Big(\big\langle (f*g)^{\otimes n}, T_s^{\otimes n}\varphi^{\otimes
n}\big\rangle\Big)=\Big(\big((f*g)\star\varphi\big)^{\otimes n}\Big)\\
=& \Big(\big(f\star (g\star\varphi)\big)^{\otimes n}\Big)=
\Big(\big\langle f, T_s(g\star\varphi)\big\rangle^{\otimes n}\Big)
=\Big(\big\langle f^{\otimes n}, T_s^{\otimes n}(g\star\varphi)^{\otimes
n}\big\rangle\Big)
=K_{\boldsymbol f}^{\otimes} K_{\boldsymbol g}^{\otimes}\boldsymbol p.
\end{split}
\]

Using \eqref{fgstarphi}, we obtain
\[
\begin{split}
K_{\boldsymbol f}^{\otimes} T^{\otimes}_s\boldsymbol p=&
\big(f^{\otimes n}\big)\star
\big(T_s^{\otimes n}\varphi^{\otimes n}\big)=
\big(f^{\otimes n}\big)\star
\big((T_s\varphi)^{\otimes n}\big)\\
=&\big((f\star T_s\varphi)^{\otimes n}\big)
=\big((T_s(f\star\varphi))^{\otimes n}\big)
=\big(T_s^{\otimes n}(f\star \varphi)^{\otimes n}\big)\\
=&\big(T_s^{\otimes n}K_{f_n}^{\otimes n}\varphi^{\otimes n}\big)=
T^{\otimes}_s K_{\boldsymbol f}^{\otimes}\boldsymbol p
\end{split}
\]
for all $s\in\mathbb R_+$.
Hence, the operator $K_{\boldsymbol f}^{\otimes}$ belongs to $\big[T^{\otimes}\big]^c$ for
all $\boldsymbol f\in\Gamma(\mathcal S'_+)$.

Conversely, let $K\in {L}(\mathcal{S}_+)$ be an operator such that
$K^{\otimes}\in\big[T^{\otimes}\big]^c$.  Let us show that there exists
$\boldsymbol h\in\Gamma(\mathcal S'_+)$ such that
$K^{\otimes}=K_{\boldsymbol h}^{\otimes}$.  Such an element is
$\boldsymbol h:=(1,h,\dots,h^{\otimes n},\dots)$, where the
distribution $h\in\mathcal S'_+$ is defined by the relation $\langle
h, \varphi\rangle:=(K\varphi)(0)$, $\varphi\in\mathcal S_+$.  Since
$(h\star\varphi)(s)=\langle h, T_s\varphi
\rangle=(KT_s\varphi)(0)=(K\varphi)(s)$, we obtain
\[
K_{\boldsymbol h}^{\otimes}\boldsymbol p=
\big((h\star\varphi)^{\otimes n}\big)=
\big((K\varphi)^{\otimes n}\big)=
\big(K^{\otimes n}\varphi^{\otimes n}\big)=
K^{\otimes}\boldsymbol p.
\]
Thus, $K^\otimes=K_{\boldsymbol h}^{\otimes}$.
So, the range of the mapping
$\Gamma(\mathcal S'_+) \ni \boldsymbol f \longmapsto
K_{\boldsymbol f}^{\otimes}
\in  L\big(\Gamma(\mathcal S_+)\big)$ coincides with the commutant
$\big[T^{\otimes}\big]^c$.
\end{proof}

Let $D$ mean the differential operator on $\mathcal S_+$. We use the
same letter $D$ to denote the operator of generalized differentiation
on $\mathcal S'_+$, i.e. $\langle Df,\varphi\rangle=-\langle f,
D\varphi\rangle$.

Let us define the operator ${\mathbb D}\in{L}\big(\Gamma(\mathcal S'_+)\big)$
as follows
\begin{equation*}
\begin{array}{cccc}
{\mathbb D}:
&
\displaystyle\Gamma(\mathcal S'_+)
&
\quad\longrightarrow\quad
&
\displaystyle\Gamma(\mathcal S'_+)\\
&
\displaystyle (1,f,\dots,f^{\otimes n},\dots)
&
\quad\longmapsto\quad
&
\displaystyle \Big(0,Df,\dots,\sum\limits_{j=1}^n
f^{\otimes(j-1)}\hat\otimes\, D f\,\hat\otimes\, f^{\otimes(n-j)},\dots\Big).
\end{array}
\end{equation*}
Its restriction onto $\Gamma(\mathcal S_+)$ acts as
\begin{equation*}
\begin{array}{cccc}
{\mathbb D}:
&
\displaystyle\Gamma(\mathcal S_+)
&
\quad\longrightarrow\quad
&
\displaystyle\Gamma(\mathcal S_+)\\
&
\displaystyle (1,\varphi,\dots,\varphi^{\otimes n},\dots)
&
\quad\longmapsto\quad
&
\displaystyle \Big(0,D\varphi,\dots,\sum\limits_{j=1}^n
\varphi^{\otimes(j-1)}\hat\otimes\, D \varphi\,\hat\otimes\, \varphi^{\otimes(n-j)},\dots\Big).
\end{array}
\end{equation*}
Analogically as in \cite{LopSharTopology} it is easy to prove that
${\mathbb D}$ is a continuous derivative.

\begin{proposition}\label{prop33}
For any $\boldsymbol f\in \Gamma(\mathcal S'_+)$ and $\boldsymbol
p\in \Gamma(\mathcal S_+)$ the following equality  holds:
\[({\mathbb D} \boldsymbol f)\star \boldsymbol p=
-\boldsymbol f\star ({\mathbb D}\boldsymbol p).\]
\end{proposition}

\begin{proof}
For any $\boldsymbol f=\big(f^{\otimes n}\big)\in\Gamma(\mathcal S'_+)$ with
$f\in\mathcal S'_+$ and
$\boldsymbol p=\big(\varphi^{\otimes n}\big)\in\Gamma(\mathcal S_+)$ with $\varphi\in\mathcal S_+$
we have
\[
\begin{split}
({\mathbb D} \boldsymbol f)\star \boldsymbol p
=&\Big(0,Df\star\varphi,\dots,\sum\limits_{j=1}^n
(f\star\varphi)^{\otimes(j-1)}\hat\otimes\, (Df\star\varphi)\,
\hat\otimes\, (f\star\varphi)^{\otimes(n-j)},\dots\Big)\\
=&-\Big(0,f\star D\varphi,\dots,\sum\limits_{j=1}^n
(f\star\varphi)^{\otimes(j-1)}\hat\otimes\, (f\star D\varphi)\,
\hat\otimes\, (f\star\varphi)^{\otimes(n-j)},\dots\Big)\\
=&-\Big(0,\big\langle f, T_s D\varphi\big\rangle,\dots,\big\langle f^{\otimes n},\sum_{j=1}^n
(T_s\varphi)^{\otimes(j-1)} \hat\otimes\, (T_s D\varphi)\,
\hat\otimes\, (T_s\varphi)^{\otimes(n-j)}\big\rangle,\dots\Big)\\
=&-\Big(0,\big\langle f, T_s D\varphi\big\rangle,\dots,\big\langle f^{\otimes n}, T_s^{\otimes n}\sum_{j=1}^n
\varphi^{\otimes(j-1)} \hat\otimes\, D\varphi\, \hat\otimes\,
\varphi^{\otimes(n-j)}\big\rangle,\dots\Big)\\
=&-\boldsymbol f\star ({\mathbb D}\boldsymbol p).
\end{split}
\]
The proposition is proved.
\end{proof}

\section{Fourier transform of polynomial tempered distributions}\label{sec.alebra}

Since each element of the space $\mathcal S_+$ may be considered
as a function $\varphi\in{L}^1(0,\infty)\cap {L}^2(0,\infty)$, we define
the Fourier transform and its inverse, as follows:
\begin{align*}
&\mathcal F_+:\mathcal S_+\ni\varphi\longmapsto\widehat{\varphi}(\xi):=\int_{\mathbb R_+}
e^{-\mathfrak{i}t\xi}{\varphi}(t)\,dt,\qquad
\xi\in\mathbb R,\\
&\mathcal F_+^{-1}:\widehat{\varphi}\longmapsto\varphi(t)=\frac{1}{2\pi}\int_{\mathbb R}
e^{\mathfrak{i}t\xi}\widehat{\varphi}(\xi)\,d\xi,
\qquad t\in\mathbb R_+.
\end{align*}

Let $\widehat{\mathcal S}_+:=\mathcal F_+[\mathcal S_+]$ stand for
the range of $\mathcal S_+$ under the map $\mathcal F_+$.
It is known \cite{Akhiezer} that $\widehat{\mathcal S}_+\subset{L}^2(\mathbb R)$.
Using the injectivity of $\mathcal F_+$, we endow the space
$\widehat{\mathcal S}_+$
with a topology induced by the topology in $\mathcal S_+$.
Therefore,
$\widehat{\mathcal S}_+$ is a nuclear
$(F)$ space (see \cite{Schaefer}).
For the strong duals the appro\-pri\-ate adjoint transform
$(\mathcal F_+^{-1})'\colon \mathcal S'_+\longmapsto {\widehat{\mathcal S}_+}'$ is well
defined. The mapping
\[\mathcal F'_+:=2\pi (\mathcal F_+^{-1})'\colon\mathcal S'_+\ni
f\longmapsto \widehat f\in\widehat{\mathcal S}_+'\]
is called the generalized Fourier transform of distributions from $\mathcal S_+'$.
The space $\widehat{\mathcal S}'_+$ is a nuclear $(DFS)$ space
as a strong dual of the nuclear $(FS)$ space
$\widehat{\mathcal S}_+$ (see \cite{Schaefer}).

Since delta functional is a unit element in the convolution
algebra $\mathcal S^\prime_+$ (see \cite{Vladimirov}), we obtain
$\widehat{\delta*f}=\widehat{f}=\widehat{f*\delta}$ and
the space $\widehat{\mathcal S}^\prime_+$ is a commutative multiplicative
algebra with the unit
$\widehat\delta$ with respect to the multiplication
$\widehat f\,\cdot\,\widehat h:=\widehat{f*h}$, $f,h\in\mathcal S'_+.$
The following bilinear form
$$
\langle\mathcal F'_+ f, \mathcal F_+\varphi\rangle=
\langle 2\pi (\mathcal F_+^{-1})' f, \mathcal F_+\varphi\rangle=
2\pi\langle f, \mathcal F_+^{-1}\mathcal F_+\varphi\rangle=2\pi\langle f,
\varphi\rangle,
$$
with
$f\in\mathcal S'_+$, $\varphi\in\mathcal S_+$, defines the
new duality $\langle\widehat{\mathcal{S}}'_+,\widehat{\mathcal S}_+\rangle$.

Denote $\Gamma(\widehat{\mathcal S}_+):=
\mathop{\bigoplus\nolimits_{\operatorname{fin}}}\limits_{\hskip-10pt
n\in\mathbb Z_+} \widehat{\mathcal S}_+^{\hat\otimes n}$ and
$\Gamma(\widehat{\mathcal S}'_+):= \TIMES\limits_{n\in\mathbb
Z_+}\widehat{\mathcal S}_+^{\prime\hat\otimes n}$. For any
elements $\boldsymbol {\widehat f}=\big(\widehat f^{\otimes
n}\big)$, $\boldsymbol {\widehat h}=\big(\widehat h^{\otimes
n}\big)\in \Gamma(\widehat{\mathcal S}'_+)$ with $f,h\in\mathcal
S'_+$ we define the operation
\[
\boldsymbol {\widehat f}\,\widehat\circledast\, \boldsymbol {\widehat h}:=
\big((\widehat f\,\cdot\,\widehat h)^{\otimes n}\big)
\]
and extend it to the whole space $\Gamma(\widehat{\mathcal S}'_+)$ by
linearity and continuity. It is obvious, that
$\Gamma(\widehat{\mathcal S}'_+)$ is an algebra relative to the
operation $\widehat\circledast$. Similarly as above, we can induce
this operation on the space $\Gamma(\widehat{\mathcal S}_+)$ that
becomes an algebra too. From \cite[Proposition~2.1]{LopSharTopology}
it follows that there exist the linear topological isomorphisms of
algebras
\[
\widehat\Upsilon:{\mathcal P}(\widehat{\mathcal S}'_+)
\longrightarrow \Gamma(\widehat{\mathcal S}_+) \quad
\text{and}\quad \widehat\Psi:{\mathcal P}'(\widehat{\mathcal
S}'_+) \longrightarrow \Gamma(\widehat{\mathcal S}'_+).
\]

Using Proposition~\ref{propTensStruct} we can extend
the map $\mathcal F_+$ onto the space $\Gamma(\mathcal S_+)$ as
follows. First of all, for any
$\varphi^{\otimes n}\in\mathcal{S}_+^{\hat\otimes n}$
with $\varphi\in\mathcal S_+$ we define the
operation
$\mathcal F_+^{\otimes_n}$ by the relations
\[
\mathcal F_+^{\otimes_n}: \varphi^{\otimes n}\longmapsto \widehat{\varphi}^{\otimes n}
\quad\text{and}\quad \mathcal F_+^{\otimes_0}:=I_{\mathbb C}.
\]
Next, we extend the mapping $\mathcal F_+^{\otimes_n}$ to the whole
space $\mathcal{S}_+^{\hat\otimes n}$ by linearity and continuity, so
$\mathcal F_+^{\otimes_n} \in{L}\big(\mathcal{S}_+^{\hat\otimes n},
\widehat{\mathcal S}_+^{\hat\otimes n}\big)$.  Finally, $\mathcal
F_+^{\otimes}$ is defined to be the mapping
\[
\mathcal F_+^{\otimes}=\big(\mathcal F_+^{\otimes_n}\big) : \Gamma(\mathcal S_+)\ni
\boldsymbol p=\big(p_n\big)\quad \longmapsto\quad
\widehat{\boldsymbol p}:=\big(\widehat p_n\big)
\in\Gamma(\widehat{\mathcal S}_+),
\]
where $\widehat p_n:= \mathcal F_+^{\otimes_n}p_n$.
It is easy to check that $\mathcal F_+^{\otimes}$ is a homomorphism
of the corresponding algebras.

\vspace*{-2pt}
\begin{remark}\label{infinitemanyvar}
Note that $\widehat{\varphi}^{\otimes n}$ for any $n\in\mathbb N$ may be treated as a
function of $n$ variables
$
\mathbb R^n\ni(\xi_1,\dots,\xi_n)\longmapsto
\widehat{\varphi}(\xi_1)\cdot\ldots\cdot\widehat{
\varphi}(\xi_n)\in \mathbb C
$
and may be written in the following way:
\[
\widehat{\varphi}^{\otimes n}(\xi_1,\dots,\xi_n)=
\int_{\mathbb R^n_+}
e^{-\mathfrak{i}(t,\xi)_n}{\varphi}
(t_1)\cdot\ldots\cdot
{\varphi}(t_n)\,dt,
\]
where $(t,\xi)_n:=t_1\xi_1+\dots+t_n\xi_n$, $dt:=dt_1\dots dt_n$.
So, elements of $\Gamma(\widehat{\mathcal S}_+)$
can be considered as functions of infinitely many variables
\begin{equation}\label{eq.infinitemanyvar}
\widehat{\boldsymbol p} : (\xi_1,\dots,\xi_n,\dots)\longmapsto
\big(\widehat p_0,\widehat p_1(\xi_1),\widehat p_2(\xi_2,\xi_3),\dots,
\widehat p_n(\xi_{\mathfrak b_n},\dots,\xi_{\mathfrak e_n}),
\dots\big),
\end{equation}
where $\mathfrak b_n:=\frac{n(n-1)}{2}+1$, $\mathfrak e_n:=\frac{n(n+1)}{2}$.
But we note that actually each $\widehat{\boldsymbol p}\in\Gamma(\widehat{\mathcal S}_+)$
depends on a finite (depending on $\widehat{\boldsymbol p}$) number of variables,
because for each $\widehat{\boldsymbol p}$ the sequence in the right-hand side of
\eqref{eq.infinitemanyvar} is finite.
\end{remark}
\vspace*{-4pt}

Define the operator $\mathcal F^{\prime\otimes}_+$ as follows
\[
\mathcal F^{\prime\otimes}_+:=(\mathcal F^{\prime\otimes_n}_+):\Gamma(\mathcal S'_+)
\ni\boldsymbol f=\big(f_n\big)
\quad\longmapsto\quad
\boldsymbol {\widehat f}:=\big(\widehat f_n\big)
\in\Gamma(\widehat{\mathcal S}'_+),
\]
where
$\widehat f_n:=\mathcal F^{\prime\otimes_n}_+f_n
\in\widehat{\mathcal S}_+^{\prime\hat\otimes n}$,
$\mathcal F^{\prime\otimes_0}_+:=I_{\mathbb C}$,
and each operator
$\mathcal F^{\prime\otimes_n}_+:\mathcal S_+^{\prime\hat\otimes n}
\longrightarrow
\widehat{\mathcal S}_+^{\prime\hat\otimes n}$,
$n\in\mathbb N$, is defined as a linear and continuous extension of the map
$f^{\otimes n} \longmapsto (\mathcal F'_+ f)^{\otimes n}$
with $f\in \mathcal S'_+$.

\vspace*{-4pt}
\begin{remark}\label{rem2}
  From Remark \ref{infinitemanyvar} it follows that $\widehat
  f_n\in\widehat{\mathcal S}_+^{\prime\hat\otimes n}\simeq
  (\widehat{\mathcal S}_+^{\hat\otimes n})'$ is a functional of $n$
  ``variables''
\[\widehat p_n(\xi_1,\dots,\xi_n)\longmapsto \langle \widehat f_n, \widehat p_n
\rangle:=\langle
\widehat f_n(\xi_1,\dots,\xi_n),
\widehat p_n(\xi_1,\dots,\xi_n)\rangle\in\mathbb C
\]
with $\widehat p_n \in\widehat{\mathcal S}_+^{\hat\otimes n}$. So,
any $\boldsymbol {\widehat f}=\big(\widehat f_n\big)\in
\Gamma(\widehat{\mathcal S}'_+)$ we consider as a functional of
infinitely many ``variables'' in the following sense (cf.
\eqref{eq.infinitemanyvar}):
\begin{equation*}
\begin{array}{cccc}
\displaystyle \boldsymbol {\widehat f}:
&
\displaystyle\Gamma(\widehat{\mathcal S}_+)
&
\longrightarrow
&
\displaystyle\mathbb C
\\[2pt]
&
\displaystyle \widehat{\boldsymbol p}(\xi_1,\dots,\xi_n,\dots)=
\big(\widehat p_n(\xi_{\mathfrak b_n},\dots,\xi_{\mathfrak e_n})\big)
&
\longmapsto
&
\displaystyle \langle \boldsymbol {\widehat f}, \widehat{\boldsymbol p} \rangle:=
\sum_{n\in\mathbb Z_+} \langle \widehat f_n, \widehat p_n \rangle.
\end{array}
\end{equation*}
\end{remark}

\section{Infinite parameter operator semigroups}\label{OC}

Let $E$ be a complex Banach space.  Let $\mathbf A:=(\mathbf A_1,
\mathbf A_2, \dots, \mathbf A_n, \dots)$ be a countable system of
operators, acting on $E$.  It is convenient for us to rewrite this
system as follows.  Denote $A_n:=(\mathbf A_{\mathfrak b_n}, \dots,
\mathbf A_{\mathfrak e_n})$, where $\mathfrak
b_n:=\frac{n(n-1)}{2}+1$, $\mathfrak e_n:=\frac{n(n+1)}{2}$.  Let by
definition $A_{0}:=\emptyset$.  Then the countable system of operators
$\mathbf A$ can be represented as $\mathbf A=(A_0, A_1, A_2, \dots,
A_n, \dots)$ or $\mathbf A=\big(A_n\big)$ for short.

For any $t\in\mathbb R^n_+$ let us denote $t\cdot A_n:=t_1 \mathbf
A_{\mathfrak b_n}+\dots+t_n \mathbf A_{\mathfrak e_n}.$ Let $A_n$
be a generator (see \cite{ButzerBerens,HillePhillips57}) of
$n$-parameter $C_0$-semigroup $\mathbb R^n_+\ni t\longmapsto
e^{-\mathfrak{i}t\cdot A_n}\in L(E)$, satisfying the condition
\begin{equation}\label{contrcond}
\sup_{t\in\mathbb R^n_+}\|e^{-\mathfrak{i}t\cdot A_n}\|_{ L(E)}\le1.
\end{equation}
In what follows we assume that operators of the set $A_n$ for all
$n\in\mathbb N$ commute with each other. Note that in this case
the semigroup can be represented (see
\cite{ButzerBerens,HillePhillips57}) as a composition of commuting
one-parameter marginal semigroups
\[
e^{-\mathfrak{i}t\cdot A_n}=e^{-\mathfrak{i}t_1 \mathbf A_{\mathfrak b_n}}
\circ\dots\circ
e^{-\mathfrak{i}t_n \mathbf A_{\mathfrak e_n}}.
\]

Let $\mathcal G$ be the set of countable systems of such generators.
For all $n\in\mathbb N$ let $\mathcal G_n$ be a set of collections of
some $n$ generators of one-parameter $C_0$-semigroups satisfying the
condition \eqref{contrcond}, and let $\mathcal G_0:=\{\emptyset\}$ by
definition.

Define the mapping
\begin{equation}\label{mapL}
\mathcal{L}:=({\mathcal L}_n)\colon
\Gamma(\mathcal S_+)\ni \boldsymbol p=\big(p_n\big)
\quad\longmapsto\quad
\widetilde{\boldsymbol p}:=\sum_{n\in\mathbb Z_+}
\widetilde p_n\in\widetilde{\mathcal S},
\end{equation}
where $\widetilde{\mathcal S}:=\sum_{n\in\mathbb Z_+}
\widetilde{\mathcal S}_n$. Here each
$\widetilde{\mathcal S}_n$, $n\in\mathbb Z_+$, is defined to be the space
of functions
\begin{equation}\label{opfuncn}
\widetilde p_n:\mathcal G_n\ni A_n
\longmapsto
\widetilde p_n(A_n):=\int_{\mathbb R^n_+}
e^{-\mathfrak{i}t\cdot A_n}p_n(t)\,dt
\in L(E)
\end{equation}
for $n\in\mathbb N$, and $\widetilde p_0\colon\mathcal G_0\ni A_0
\longmapsto\widetilde p_0(A_0):=p_0I_E\in L(E)$, where the integral is
understood in the sense of Bochner.

If the assumption \eqref{contrcond} holds, then all the mappings
${\mathcal L}_n: p_n\longmapsto \widetilde p_n$, $n\in\mathbb Z_+$,
are isomorphisms by virtue of \cite[Theorem 15.2.1]{HillePhillips57}.
Indeed, the semigroups $\{e^{-\mathfrak{i}(\lambda, t)}I_E:t\in\mathbb R^n_+\}$ with
$-\mathop{\rm Im}\lambda\in\mathop{\rm int}\mathbb{R}^n_+$ satisfy the condition \eqref{contrcond}.
Therefore, their generators $(-\mathfrak{i}\lambda_1 I_E,\dots,-\mathfrak{i}\lambda_n I_E)$
belong to $\mathcal G_n$.
Note that
\[
\widetilde p_n(-\mathfrak{i}\lambda_1 I_E,\dots,-\mathfrak{i}\lambda_n I_E)=\int_{\mathbb R^n_+}
e^{-\lambda\cdot t} p_n(t)\,dt
\]
is the Laplace transform of a function
$p_n\in\mathcal{S}_+^{\hat\otimes n}$.
Particularly, it follows that
if $\widetilde p_n\equiv0$, then $p_n\equiv0$,
i.e., $\mathop{\rm Ker}\mathcal{L}_n=\{0\}$,
$n\in\mathbb N$. Hence, $\mathop{\rm Ker}\mathcal{L}=\{0\}$ and the map
$\mathcal{L}$ is an isomorphism.

\begin{remark}\label{rem41}
The mapping
$\mathcal{L}\colon
\Gamma(\mathcal S_+) \longrightarrow \widetilde{\mathcal S}$
is a homomorphism of the algebra
$\big\{\Gamma(\mathcal S_+),\circledast\big\}$
and an algebra of operator valued functions defined on
$\mathcal G$. On the other hand, the map
$\mathcal F_+^{\otimes} : \Gamma(\mathcal S_+) \longrightarrow
\Gamma(\widehat{\mathcal S}_+)$ is a homomorphism too.
So, we can treat the mapping
\[
\mathcal{L}\circ(\mathcal F_+^{\otimes})^{-1}\colon
\Gamma(\widehat{\mathcal S}_+)
\longrightarrow
\widetilde{\mathcal S}
\]
as an ``elementary'' functional calculus.  In other words, we
understand the operator $\widetilde{\boldsymbol p}(\mathbf A)
=\sum_{n\in\mathbb Z_+}\widetilde p_n(A_n) \in L(E)$ as a ``value'' of
a function $\widehat{\boldsymbol p}$ of infinitely many variables (see
\eqref{eq.infinitemanyvar}) at a countable system $\mathbf A:=(\mathbf
A_1, \mathbf A_2, \dots, \mathbf A_n, \dots)\in\mathcal G$ of
generators of contraction $C_0$-semigroups.
\end{remark}

Consider the one-parameter semigroup $\widetilde T^{\otimes}:\mathbb
R_+\ni s\longmapsto \widetilde T_{s}^{\otimes}\in
L\big(\widetilde{\mathcal S}\,\big)$ on the space $\widetilde{\mathcal
  S}$, where
\begin{equation*}
\widetilde T_{s}^{\otimes}:=\big(\widetilde T_{s}^{\otimes n}\big)\colon
\widetilde{\boldsymbol p}=\sum_{n\in\mathbb Z_+}\widetilde p_n
\quad\longmapsto\quad
\widetilde T_{s}^{\otimes}\widetilde{\boldsymbol p}:=
\sum_{n\in\mathbb Z_+}
\widetilde T_{s}^{\otimes n}\widetilde p_n.
\end{equation*}
The function $\widetilde T_{s}^{\otimes n}\widetilde p_n\in\widetilde{\mathcal S}_n$
is defined to be the map
\begin{equation*}
\widetilde T_{s}^{\otimes n}\widetilde p_n\colon
A_n
\longmapsto
\widetilde T_{s}^{\otimes n}\widetilde p_n (A_n):=\int_{\mathbb R^n_+}
e^{-\mathfrak{i}t\cdot A_n}p_n(t+s)\,dt.
\end{equation*}
Here the function $\widetilde p_n$ of operator argument is defined by \eqref{opfuncn}.

Using Bochner's integral properties  (see
\cite[3.7]{HillePhillips57}), we obtain that for any $\boldsymbol
p=\big(p_n\big)\in \Gamma(\mathcal S_+)$ with
$p_n=\varphi^{\otimes n}\in{\mathcal S}_+^{\hat\otimes n}$,
$\varphi\in{\mathcal S}_+$, the following equalities
\[
\begin{split}
\widetilde{T_{s}^{\otimes}\boldsymbol p}(\mathbf A)=
\mathcal{L}\big[\big(T_{s}^{\otimes n}p_n\big)\big](\mathbf A)&=
\mathcal{L}\big[\big((T_s\varphi)^{\otimes n}\big)\big](\mathbf A)\\
&=I_E+\sum_{n\in\mathbb N}\int_{\mathbb R^n_+}
e^{-\mathfrak{i}t\cdot A_n}{\varphi}
(t_1+s)\cdot\ldots\cdot{\varphi}(t_n+s)\,dt\\
&=\widetilde p_0(A_0)+\sum_{n\in\mathbb N}
\widetilde T_{s}^{\otimes n}\widetilde p_n(A_n)
=\widetilde T_{s}^{\otimes}\widetilde{\boldsymbol p}(\mathbf A)
\end{split}
\]
hold for all $s\in\mathbb R_+$ and
$\mathbf A:=\big(A_n\big)\in\mathcal G$.

Hence, the operator $\widetilde T_{s}^{\otimes}$ can be represented
as follows:
$\widetilde T_{s}^{\otimes}=\mathcal{L}\circ T_{s}^{\otimes}\circ\mathcal{L}^{-1}$.
Continuity of the mappings $T_{s}^{\otimes}$ and $\mathcal{L}$ as well as
openness of $\mathcal{L}$ imply that
the semigroup $\widetilde T^\otimes:\mathbb R_+\ni
s\longmapsto \widetilde T_s^\otimes\in{L}\big(\widetilde{\mathcal S}\,\big)$
has the $C_0$-property.

We define commutant of the semigroup $\widetilde T^\otimes$ to be the set
\[
\big[\widetilde T^\otimes\big]^c:=\big\{\widetilde T\in{L}\big(\widetilde{\mathcal S}\,\big)
\colon \widetilde T\circ\widetilde T^\otimes_s=\widetilde T^\otimes_s\circ \widetilde T, \forall
s\in\mathbb R_+ \big\}.
\]

Define the mapping
\begin{equation}\label{PhiA}
\varPhi:=\big(\varPhi_{n}\big) :
\Gamma({\mathcal S}'_+)\ni
\boldsymbol f=\big(f_n\big)
\quad\longmapsto\quad
\varPhi_{\boldsymbol f}:=\sum_{n\in\mathbb Z_+}
\varPhi_{f_n}\in
 L\big(\widetilde{\mathcal S}\,\big),
\end{equation}
where
${f}_n:={f}^{\otimes n}\in\mathcal S_+^{\prime\hat\otimes n}$,
$f\in\mathcal S'_+$.
Here $\varPhi_{f_n}\in L\big(\widetilde{\mathcal S}_{n}\,\big)$,
$n\in\mathbb Z_+$, is defined by the following formulas:
$(\varPhi_{f_0}\widetilde p_0)(A_0):=I_E$ and
\[
\varPhi_{f_n}: \widetilde p_n \longmapsto \widetilde
q_n:=\varPhi_{f_n} \widetilde p_n,\quad \text{where}\quad
\widetilde q_n(A_n):=\int_{\mathbb R^n_+} e^{-\mathfrak{i}t\cdot
A_n}K_f^{\otimes n}p_n(t)\,dt,\quad n\in\mathbb N.
\]
Here the function $\widetilde p_n$ of operator argument is defined by
\eqref{opfuncn}, and the operator $K_f^{\otimes n}$ is defined by
\eqref{Kotimes} and \eqref{Kfotimesn}.

\begin{theorem}\label{main}
  The map $\varPhi$ defined by \eqref{PhiA} is an algebraic
  isomorphism of the algebra $\big\{\Gamma({\mathcal
    S}'_+),\circledast\big\}$ and the subalgebra in the commutant
  $\big[\widetilde T^\otimes\big]^c$ of operators of the form
  $\widetilde K^{\otimes}=\mathcal{L}\circ
  K^{\otimes}\circ\mathcal{L}^{-1}\in{L}\big(\widetilde{\mathcal
    S}\,\big)$, where $K\in{L}({\mathcal S}_+)$.  In particular, the
  equality $\varPhi_{\boldsymbol f\circledast \boldsymbol g}=
  \varPhi_{\boldsymbol f}\circ\varPhi_{\boldsymbol g}$ holds for all
  $\boldsymbol f, \boldsymbol g\in\Gamma({\mathcal S}'_+)$ and
  $\varPhi_{\boldsymbol\delta}$ is the identity in $
  L\big(\widetilde{\mathcal S}\,\big)$, where
  $\boldsymbol\delta=\big(\delta^{\otimes n}\big)$.

  Moreover, differential the property
\begin{equation}\label{diffprop}
\varPhi_{{\mathbb D} \boldsymbol f}\widetilde {\boldsymbol p}=
-\varPhi_{\boldsymbol f}\widetilde{{\mathbb D}\boldsymbol p}
\end{equation}
holds for any $\boldsymbol f\in\Gamma({\mathcal S}'_+)$ and
$\boldsymbol p\in\Gamma(\mathcal S_+)$.
\end{theorem}

\begin{proof}
  For any $\boldsymbol f=\big(f_n\big)\in\Gamma({\mathcal S}'_+)$,
  where $f_n:=f^{\otimes n}$ with $f\in\mathcal S'_+$, and
  $\boldsymbol p=\big(p_n\big)\in \Gamma({\mathcal S}_+)$ the
  equalities
\begin{equation}\label{eq18}
\begin{split}
(\varPhi_{\boldsymbol f} \widetilde{\boldsymbol p})(\mathbf A)=
\sum_{n\in\mathbb Z_+}(\varPhi_{f_n} \widetilde p_n)(A_n)
&=I_E+\sum_{n\in\mathbb N}\int_{\mathbb R^n_+}
e^{-\mathfrak{i}t\cdot A_n}K_f^{\otimes n}p_n(t)\,dt\\
&=\mathcal{L}\big[\big(K_f^{\otimes n}p_n\big)\big](\mathbf A)
=\widetilde{K_{\boldsymbol f}^{\otimes}\boldsymbol p}(\mathbf A)
\end{split}
\end{equation}
are valid for all $\mathbf A:=\big(A_n\big)\in\mathcal G$. It follows
that the map $\varPhi$ can be represented in the form
$\varPhi_{\boldsymbol f}= \mathcal{L}\circ K_{\boldsymbol
  f}^{\otimes}\circ\mathcal{L}^{-1}$. Continuity of the mappings
$K_{\boldsymbol f}^{\otimes}$ and $\mathcal{L}$ as well as openness of
$\mathcal{L}$ imply that $\varPhi_{\boldsymbol f}\in
L\big(\widetilde{\mathcal S}\,\big)$ for all $\boldsymbol
f\in\Gamma({\mathcal S}'_+)$. It follows that the equalities
\[
\begin{split}
(\varPhi_{\boldsymbol f}\widetilde T^\otimes_s\widetilde{\boldsymbol p})(\mathbf A)=
\sum_{n\in\mathbb Z_+}
(\varPhi_{f_n}\widetilde T_{s}^{\otimes n}\widetilde p_n)(A_n)&=
I_E+\sum_{n\in\mathbb N}\int_{\mathbb R^n_+}
e^{-\mathfrak{i}t\cdot A_n}K_f^{\otimes n}T_{s}^{\otimes n}p_n(t)\,dt
\\
&=I_E+\sum_{n\in\mathbb N}\int_{\mathbb R^n_+}
e^{-\mathfrak{i}t\cdot A_n}T_{s}^{\otimes n}K_f^{\otimes n}p_n(t)\,dt\\
&=I_E+\sum_{n\in\mathbb N}\widetilde T_{s}^{\otimes n}\int_{\mathbb R^n_+}
e^{-\mathfrak{i}t\cdot A_n}K_f^{\otimes n}p_n(t)\,dt\\
&=\sum_{n\in\mathbb Z_+}
(\widetilde T_{s}^{\otimes n}\varPhi_{f_n}\widetilde p_n)(A_n)=
(\widetilde T^\otimes_s\varPhi_{\boldsymbol f}\widetilde{\boldsymbol p})(\mathbf A)
\end{split}
\]
hold for all $s\in\mathbb R_+$, $\widetilde{\boldsymbol p}=
\sum_{n\in\mathbb Z_+} \widetilde p_n\in\widetilde{\mathcal S}$
and
$\mathbf A:=\big(A_n\big)\in\mathcal G$.
Hence, for all $\boldsymbol {f}\in\Gamma({\mathcal S}'_+)$
the operator $\varPhi_{\boldsymbol f}$ belongs to the commutant
$\big[\widetilde T^\otimes\big]^c$.

Conversely, let $\widetilde K^{\otimes}=\mathcal{L}\circ
K^{\otimes}\circ\mathcal{L}^{-1}\in{L}\big(\widetilde{\mathcal
  S}\,\big)$ with $K\in{L}({\mathcal S}_+)$ belong to the commutant
$\big[\widetilde T^\otimes\big]^c$. Then
\[
\begin{split}
\mathcal{L}\circ
K^{\otimes}\circ T_s^{\otimes}\circ\mathcal{L}^{-1}=&
\mathcal{L}\circ
K^{\otimes}\circ\mathcal{L}^{-1}\circ\mathcal{L}\circ T_s^{\otimes}\circ\mathcal{L}^{-1}=
\widetilde K^{\otimes}\circ \widetilde T_s^{\otimes}
=\widetilde T_s^{\otimes}\circ \widetilde K^{\otimes}\\
=&\mathcal{L}\circ T_s^{\otimes}
\circ\mathcal{L}^{-1}\circ\mathcal{L}\circ K^{\otimes}\circ\mathcal{L}^{-1}=\mathcal{L}\circ
T_s^{\otimes}\circ K^{\otimes} \circ\mathcal{L}^{-1},
\end{split}
\]
therefore the operator $K^{\otimes}$ belongs to the commutant of
the semigroup $T_s^{\otimes}$. From the proof of Theorem
\ref{theorCrossCorel} it follows that there exists a unique
$f\in{\mathcal S}'_+$ such that $K=K_f$ and
$K^{\otimes}=K_{\boldsymbol f}^{\otimes}$ with $\boldsymbol
f=\big(f^{\otimes n}\big)$. Hence, $\widetilde
K^{\otimes}=\widetilde K_{\boldsymbol f}^{\otimes}$.

The proved above property, $K_{\boldsymbol f\circledast\boldsymbol
  g}=K_{\boldsymbol f}\circ K_{\boldsymbol g}$, implies the
equality $K_{\boldsymbol f\circledast \boldsymbol
  g}^{\otimes}=K_{\boldsymbol f}^{\otimes}\circ K_{\boldsymbol
  g}^{\otimes}$. Therefore,
\[
\begin{split}
\varPhi_{\boldsymbol f\circledast \boldsymbol g}=
\mathcal{L}\circ K_{\boldsymbol f\circledast \boldsymbol g}^{\otimes}\circ\mathcal{L}^{-1}&=
\mathcal{L}\circ K^{\otimes}_{\boldsymbol f}\circ
K^{\otimes}_{\boldsymbol g}\circ\mathcal{L}^{-1}\\
&=\mathcal{L}\circ K^{\otimes}_{\boldsymbol f}\circ\mathcal{L}^{-1}\circ\mathcal{L}\circ
K^{\otimes}_{\boldsymbol g}\circ\mathcal{L}^{-1}=
\varPhi_{\boldsymbol f}\circ
\varPhi_{\boldsymbol g}.
\end{split}
\]

As a consequence, we obtain the equalities
$\varPhi_{\boldsymbol\delta}\circ \varPhi_{\boldsymbol f}=
\varPhi_{{\boldsymbol\delta}\circledast \boldsymbol f}=
\varPhi_{\boldsymbol f}= \varPhi_{\boldsymbol f\circledast
  {\boldsymbol\delta}}= \varPhi_{\boldsymbol f}\circ
\varPhi_{\boldsymbol\delta}$, so, $\varPhi_{\boldsymbol\delta}\in
L\big(\widetilde{\mathcal S}\,\big)$ acts as the identity operator.

It remains to prove the differential property \eqref{diffprop}.
From \eqref{eq18}
it follows $\varPhi_{\boldsymbol f}\widetilde{\boldsymbol p}=
\widetilde{\boldsymbol f\star\boldsymbol p}$. So, using the Proposition
\ref{prop33}, we obtain
\[
\varPhi_{{\mathbb D}\boldsymbol f}\widetilde{\boldsymbol p}=
\widetilde{({\mathbb D}\boldsymbol f)\star\boldsymbol p}=
-\widetilde{\boldsymbol f\star ({\mathbb D}\boldsymbol p})=
-\varPhi_{\boldsymbol f}\widetilde{{\mathbb D}\boldsymbol p}.
\]
Thus, the theorem is proved.
\end{proof}

\begin{remark}
  For any fixed $\boldsymbol p\in\Gamma({\mathcal S}_+)$ the map
  $\Gamma({\mathcal S}'_+)\ni \boldsymbol {f} \longmapsto
  \varPhi_{\boldsymbol f}\widetilde{\boldsymbol
    p}\in\widetilde{\mathcal S}$ is a homomorphism of the algebra
  $\big\{\Gamma({\mathcal S}'_+),\circledast\big\}$ and the algebra of
  operator-valued functions defined on $\mathcal G$.  Therefore we can
  treat this map as a functional calculus in the algebra of polynomial
  tempered distributions.  It is easy to see that a function
  $\varPhi_{\boldsymbol f}\widetilde{\boldsymbol p}$ of operator
  argument can be represented as $\varPhi_{\boldsymbol
    f}\widetilde{\boldsymbol p}= \widetilde{\boldsymbol f\star
    \boldsymbol p}$ (see \eqref{mapL}).  From \eqref{eq18} it follows
  that the operator $\varPhi_{\boldsymbol f}\widetilde{\boldsymbol
    p}(\mathbf A)= \widetilde{\boldsymbol f\star \boldsymbol
    p}(\mathbf A)\in L(E)$ can be understood as a ``value'' of a
  function $\widehat{\boldsymbol f\star \boldsymbol p}$ of infinite
  many variables at a countable system $\mathbf A:=(\mathbf A_1,
  \mathbf A_2, \dots, \mathbf A_n, \dots)\in\mathcal G$ of generators
  of contraction $C_0$-semigroups.
\end{remark}

\begin{example}
  Let us consider the case of a countable set of second derivative
  operators.  Let $H_n:=L^2_{sym}(\mathbb R^n)\simeq{L^2(\mathbb
    R)}^{\hat\otimes n}$, $n\in\mathbb N$, be the space of complex
  valued square integrable symmetric functions
  $y(\xi)=y(\xi_1,\dots,\xi_n)$.  Set $H_0:=\mathbb C$.  It is known
  that the symmetric Fock space $\displaystyle
  H:=\bigoplus_{n\in\mathbb Z_+}H_n$ is a Hilbert space (see
  e.g. \cite{QuantumField}).  As above, let $\mathfrak
  b_n=\frac{n(n-1)}{2}+1$, $\mathfrak e_n=\frac{n(n+1)}{2}$.  Define
  the operators $\mathbf D_{n,m}^2: H\longrightarrow H$, $n\in\mathbb
  N$, $\mathfrak b_n\le m\le \mathfrak e_n$, as follows
\[
\mathbf D_{n,m}^2:=0_{H_0}\otimes\dots\otimes\, 0_{H_{n-1}}\otimes\,
\frac{\partial^2}{\partial\xi_m^2}
\otimes\, 0_{H_{n+1}}\otimes\dots,
\]
where $0_{H_n}$, $n\in\mathbb Z_+$, denote zero operators of the corresponding spaces.

Let us define an ``elementary'' functional calculus in the algebra of
polynomial test functions for the countable set of operators
\[
\mathbf D^2:=(\mathbf D_{1,1}^2, \mathbf D_{2,1}^2, \mathbf D_{2,2}^2,\dots,
\mathbf D_{n,\mathfrak b_n}^2,\dots,\mathbf D_{n,\mathfrak e_n}^2 \dots).
\]
Let $D^2_n:=(\mathbf D_{n,\mathfrak b_n}^2, \dots, \mathbf D_{n,\mathfrak e_n}^2)$.
It is easy to see that $D_n^2$, $n\in\mathbb N$,
generates the semigroup
\[
\mathbb R^n_+\ni t=(t_1,\dots,t_n)\longmapsto
e^{-\mathfrak{i}t\cdot D_n^2}
\in L(H),
\]
where
\[
e^{-\mathfrak{i}t\cdot D_n^2}
:=I_{H_0}\otimes\ldots\otimes\, I_{H_{\mathfrak b_n-1}}\otimes\,
e^{-\mathfrak{i}t_1\frac{\partial^2}{\partial\xi_{\mathfrak b_n}^2}}
\circ\ldots\circ
e^{-\mathfrak{i}t_n\frac{\partial^2}{\partial\xi_{\mathfrak e_n}^2}}
\otimes\,
I_{H_{\mathfrak e_n+1}}\otimes\ldots
\]

Denote
\[\mathfrak{g}_n(t,\zeta):=
\prod\limits_{j=1}^n\dfrac{1}{\sqrt{4\pi t_j}}e^{-\frac{\zeta_j^2}{4t_j}}.
\]
From \cite[Example 2]{LopSharSMJ} it follows that
the semigroup $e^{-\mathfrak{i}t\cdot D_n^2}$
acts as 
\[
e^{-\mathfrak{i}t\cdot D_n^2}y=
\big(y_0,\dots,y_{n-1},\mathfrak{g}_n(-\mathfrak{i}t,\,\cdot\,)*y_n,
y_{n+1},\dots\big)
\]
for any $y=(y_0,y_1,\dots,y_n,\dots)\in H$.

Let $\boldsymbol p=\big(p_n\big)\in\Gamma(\mathcal S_+)$ be given.
If we ``substitute'' the countable set $\mathbf D^2$ of operators instead
of variables of a function $\widehat{\boldsymbol p}$
(see \eqref{eq.infinitemanyvar})  we obtain the operator
acting as 
\[
\begin{split}
\widetilde{\boldsymbol p}(\mathbf D^2)y(\xi_1,\xi_2,\dots)=&y_0+\sum_{n\in\mathbb N}
\widetilde p_n(D_n^2)y_n(\xi_{\mathfrak{b}_n},\dots,\xi_{\mathfrak{e_n}})\\
=&y_0+\sum_{n\in\mathbb N}
\int_{\mathbb R^n_+}
(\mathfrak{g}_n(-\mathfrak{i}t,\,\cdot\,)*y_n)(\xi_{\mathfrak{b}_n},\dots,\xi_{\mathfrak{e_n}})p_n(t)\,dt,
\end{split}
\]
where $y(\xi_1,\xi_2,\dots)=
\big(y_0,y_1(\xi_1),y_2(\xi_2,\xi_3),\dots,y_n(\xi_{\mathfrak{b}_n},\dots,
\xi_{\mathfrak{e}_n}),\dots\big)\in H$
is a function of infinite many variables.
\end{example}

{\it Acknowledgments}. The author would like to thank the referee
for valuable comments which helped to improve the manuscript.

\end{document}